\documentclass[11pt,a4paper]{article}

\usepackage{epsf,epsfig,amsfonts,amsgen,amsmath,amstext,amsbsy,amsopn,amsthm,cases,listings,color}
\usepackage{ebezier,eepic}
\usepackage{color}
\usepackage{multirow}
\usepackage{epstopdf}
\usepackage{graphicx}   
\usepackage{pgf,tikz}
\usepackage{mathrsfs}
\usepackage[marginal]{footmisc}
\usepackage{enumerate}
\usepackage{enumitem}
\usepackage[titletoc]{appendix}
\usepackage{booktabs}
\usepackage{url}
\usepackage{mathtools}
\usepackage{pdflscape} 
\usepackage{pgfplots}
\usepackage{authblk}
\usepackage{amssymb}
\usepackage{wasysym}
\usepackage{empheq}
\usepackage{dsfont}
\usepackage{tikz}
\usepackage{longtable}
\usepackage{float}
\usepackage{subfigure}
\pgfplotsset{compat=1.18}
\usepackage{mathrsfs}
\usepackage{wasysym} 
\usetikzlibrary{arrows}
\usepackage{aligned-overset}
\usepackage{bm}%for vector
\usepackage{bbm}%for vector
\usepackage[T1]{fontenc}%for polish hook
\usepackage{amsmath}

\allowdisplaybreaks[1]

\newtheorem{definition}{Definition} [section]
\newtheorem{theorem}[definition]{Theorem}
\newtheorem{lemma}[definition]{Lemma}
\newtheorem{proposition}[definition]{Proposition}

\newtheorem{claim}[definition]{Claim}
\newtheorem{problem}[definition]{Problem}

\begin{document}
\title{\bf\Large On cliques in hypergraphs under bounded $(j,p)$-norm\thanks{Research supported by National Key R\&D Program of China (Grant No. 2023YFA1010202). }}
\author{
Jianfeng Hou\thanks{
Emails: jfhou@fzu.edu.cn},\quad \quad 
Hongbin Zhao\thanks{
Emails: hbzhao2024@163.com },
\\
\small Center for Discrete Mathematics, Fuzhou University, Fujian, 350108, China
}
\date{}
\maketitle
%%%%%%%%%%%%%%%%%%%%%%%%%%%%%%%%%%%%%%%%%%%%%
\begin{abstract}
Let $\mathcal{H}$ be an $r$-uniform hypergraph. For $S\in \binom{V(\mathcal{H})}{j}$, let $\deg(S)$ be the number of edges of $\mathcal{H}$ containing $S$, and define the $(j,p)$-norm of $\mathcal{H}$ by
\[
\|\mathcal{H}\|_{j,p}=\left(\sum_{S\in \binom{V(\mathcal{H})}{j}}\deg(S)^p\right)^{1/p}.
\]
Motivated by a problem of Chao, Dong, Shen and Yang, we determine the maximum number of $t$-cliques in an $n$-vertex $r$-graph with bounded $(j,p)$-norm in the range $p>(t-j)/(r-j)$. The proof uses an entropy argument adapted to hypergraphs, together with a continuous interpolation step. The bound is sharp whenever the corresponding Steiner systems exist.

\medskip
\end{abstract}

%%%%%%%%%%%%%%%%%%%%%%%%%%%%%%%%%%%%%%%%%%%%%
\section{Introduction}

For a graph $G$, let $k_t(G)$ denote the number of  $t$-cliques $K_t$ in $G$. Estimating $k_t(G)$ under natural edge-density conditions is a classical problem in extremal graph theory. Tur\'an's theorem implies that any graph on $n$ vertices with more edges than the Tur\'an graph contains a copy of $K_{t}$. More generally, determining the minimum possible value of $k_t(G)$ among graphs whose edge density exceeds the Tur\'an threshold is the \emph{clique supersaturation problem}. This problem was completely resolved through the work of Razborov~\cite{R08} for triangles, Nikiforov~\cite{N11} for $4$-cliques, and Reiher~\cite{R16}, who proved the Lov\'asz--Simonovits clique density conjecture~\cite{LS83} in full generality.

Another direction is to maximize $k_t(G)$ over graphs $G$ subject to specific parameter constraints. A typical example concerns graphs with a fixed number of edges, where the extremal maximum is precisely governed by the classical Kruskal--Katona theorem~\cite{K63,K66}. At the other extreme, when both the number of vertices and the maximum degree $\Delta$ of $G$ are fixed, Gan, Loh, and Sudakov~\cite{GLS15} conjectured that $k_t(G)$ is maximized by a disjoint union of cliques of size $\Delta+1$ (together with at most one smaller clique). This conjecture, which reformulates earlier questions on independent sets~\cite{G11,EG14,CR14}, was fully resolved by Chase~\cite{C20}; subsequently, Chao and Dong~\cite{CD22} provided an alternative proof.

The number of edges and the maximum degree of a graph correspond naturally to the $\ell_1$-norm and the $\ell_\infty$-norm of its degree sequence, respectively. This suggests the study of extremal problems  under $\ell_p$-norm constraints on the degree sequence. For a real number $p > 0$, the $p$-norm of $G$, denoted by $||G||_p$ is defined as 
\begin{align*}
    ||G||_p = \left( \sum_{v \in V} \deg(v)^p \right)^{1/p}, 
\end{align*}
where $\deg(v)$ denotes the \emph{degree} of $v$ in $G$.

The systematic study of $\ell_p$-norms in extremal graph theory originated in Tur\'an-type problems. While classical Tur\'an problems ask for the maximum number of edges in an $n$-vertex graph with prescribed forbidden subgraphs, equivalently maximizing the $\ell_1$-norm of the degree sequence, Caro and Yuster~\cite{CY00,CY04} extended this framework to $\ell_p$-norms. Erd\H{o}s--Stone-type results in this setting were subsequently established by Bollob\'{a}s--Nikiforov~\cite{BN12}.

Recently, Chao, Dong, Shen, and Yang~\cite{CDSY24} investigated the maximum number of $t$-cliques in graphs with bounded $\ell_p$-norm, establishing asymptotically sharp upper bounds both with and without a fixed number of vertices.

In this note, we study the maximum number of $t$-cliques in hypergraphs with a fixed number of vertices and bounded $(j,p)$-norm. For $r\ge 2$ an $r$-uniform hypergraph (henceforth an $r$-graph) $\mathcal{H}$ is a collection of $r$-subsets of some finite set $V$. For an integer $j$ with $1\le j <r $ and a real number $p > 0$, the \emph{degree} of a $j$-subset $S \in \binom{V}{j}$, denoted by $\deg(S)$, is defined as the number of edges in $\mathcal{H}$ containing $S$. Following the definitions in \cite{CILLP24,CDSY24}, the \emph{$(j, p)$-norm} of $\mathcal{H}$ is defined as
\begin{align*}
    \|\mathcal{H}\|_{j,p} = \left( \sum_{S \in \binom{V}{j}} \deg(S)^p \right)^{1/p}.
\end{align*}
Note that if $j = r - 1$, this framework seamlessly recovers the  $\ell_p$-norm $\|\mathcal{H}\|_p$ of $\mathcal{H}$.

Similarly, the study of $(j,p)$-norms in hypergraphs originated in Tur\'an-type problems. Balogh, Clemen, and Lidick\'{y}~\cite{BCL22a,BCL22b} initiated this direction by investigating Tur\'an-type problems for $3$-graphs under $\ell_2$-norm constraints. Notably, they resolved the Tur\'an tetrahedron problem under the $\ell_2$-norm using computer--assisted flag algebra computations, a powerful tool introduced by Razborov~\cite{Raz07}. Subsequently, Chen, Iľkovi\v{c}, Le\'on, Liu, and Pikhurko \cite{CILLP24} initiated a systematic study of Tur\'an-type problems under general $(j,p)$-norm constraints. See~\cite{CL24,GLMP25,BDHLZ25,HLZ25,BCD+25,BL26,ZH26} for further results in this active direction.

A natural extension of Tur\'an-type problems is to estimate the number of $t$-cliques $k_t^r(\mathcal{H})$ in an $r$-graph $\mathcal{H}$ with bounded $(j,p)$-norm. Very recently, Chao, Dong, Shen, and Yang~\cite{CDSY24} established a tight upper bound for $k_t(\mathcal{H})$ under a fixed $(j,p)$-norm. 

To state their result, we first introduce some notation. For a real number $x\ge 0$, define
\begin{align*}
    \binom{x}{k} = \frac{x(x-1)\cdots(x-k+1)}{k!}.
\end{align*}
When $p>\frac{t-j}{r-j}$, let
\begin{align*}
    \tilde{h}(x) = \frac{\binom{x}{t}}{\binom{x}{j}\binom{x-j}{r-j}^p}.
\end{align*}
Let $\widetilde s_R$ be the unique real number maximizing $\widetilde h(x)$ on the interval $(t-1,\infty)$. As shown by Chao, Dong, Shen, and Yang \cite[Section 5]{CDSY24}, such a number exists and is unique. Moreover, $\widetilde h(x)$ is strictly increasing on $(t-1,\widetilde s_R]$ and strictly decreasing on $[\widetilde s_R,\infty)$. 

\begin{theorem}[Chao--Dong--Shen--Yang \cite{CDSY24}] \label{thm:CDSY}
Suppose $C > 0$. Let $H$ be an $r$-uniform hypergraph with $\|\mathcal{H}\|_{j,p} \le C$.
\begin{enumerate}
    \item If $0 < p \le \frac{t-j}{r-j}$, then $k_t^r(\mathcal{H}) \le \max\left\{\binom{u}{t}, 0\right\}$, where $u$ is the positive real with $C = \sqrt[p]{\binom{u}{j}\binom{u-j}{r-j}}$.
    \item If $p > \frac{t-j}{r-j}$, then $k_t^r(\mathcal{H}) \le C^p \cdot \widetilde h(\widetilde s_R)$.
\end{enumerate}
\end{theorem}

They also  remarked that it would be interesting to consider the case where $\mathcal{H}$  has a fixed number of vertices. This fixed-vertex variant is natural from the extremal point of view. 

\begin{problem}[Chao--Dong--Shen--Yang \cite{CDSY24}] \label{prob:CDSY}
For integers $r\ge 2, t> r, n$ and a real number $C$, what is the maximum possible number of $t$-cliques in an $n$-vertex $r$-graph $\mathcal{H}$ with $\|\mathcal{H}\|_{j,p}\le C$?
\end{problem}

For Problem~\ref{prob:CDSY}, when $0 < p \le \frac{t-j}{r-j}$, the upper bound  in Theorem~\ref{thm:CDSY}~(i) continues to hold under the additional constraint  that the number of vertices is  fixed at  $n$. Indeed, any meaningful value of $C$ satisfies 
$C^p \le \binom{n}{j}\binom{n-j}{r-j}^p$, which by monotonicity forces $n \ge u$. 
By the same argument, the bound in Theorem~\ref{thm:CDSY}~(ii) holds when 
$p > \frac{t-j}{r-j}$ and $C^p < \binom{n}{j} \binom{\tilde{s}_{\mathbb{R}}-j}{r-j}^p$. Thus the main unresolved part of Problem~\ref{prob:CDSY} is the high-norm range in which the fixed vertex constraint interacts nontrivially with the local $j$-degree constraint.
For the remaining case, we have

\begin{theorem}\label{thm:hypergraph_fixed_vertex}
Let $t, r, j, n$ be fixed integers such that $n\ge t > r > j \ge 1$, and let $C>0$ and $p>0$ be real numbers satisfying 
\[
p > \frac{t-j}{r-j} \quad \text{ and } \quad C^p \ge \binom{n}{j} \binom{\tilde{s}_{\mathbb{R}}-j}{r-j}^p
\]
If $\mathcal{H}$ is an $n$-vertex $r$-graph with $\|\mathcal{H}\|_{j,p}\le C$, then 
\[
k_t^{(r)}(\mathcal{H}) \le \frac{\binom{n}{j}}{\binom{u}{j}} \binom{u}{t},
\]
where $u \ge \tilde{s}_{\mathbb{R}}$ is the unique real number satisfying $\binom{n}{j}^{1/p} \binom{u-j}{r-j} = C$.
\end{theorem}

The main difficulty in proving Theorem~\ref{thm:hypergraph_fixed_vertex} is that the graph argument of Chao--Dong--Shen--Yang does not extend directly to $r$-graphs. In the hypergraph setting, the first exposed object is naturally a $j$-set rather than a single vertex, and its entropy must be controlled simultaneously by the global bound $|V(\mathcal H)|=n$ and by the $(j,p)$-norm of the local degrees. We overcome this by decomposing the entropy of a random ordered $t$-clique into a $j$-prefix, an $(r-j)$-extension, and a tail, and then using a continuous interpolation argument to connect the resulting product inequalities with the monotonicity of $\widetilde h$.

The following proposition shows that the upper bound in Theorem \ref{thm:hypergraph_fixed_vertex} is the best possible under certain number-theoretic conditions. 

\begin{proposition} \label{prop:exact_construction}
    Let $u, t, r, j$ be fixed integers such that $u> t > r > j \ge 1$. 
For sufficiently large $n$ satisfying the necessary divisibility conditions
\begin{equation}\label{necessary-divisibility-condition-n}
    \binom{n-i}{j-i} \equiv 0 \pmod{\binom{u-i}{j-i}} \quad \text{for all } i \in \{0, 1, \dots, j-1\},
\end{equation}
there exists an $n$-vertex $r$-uniform hypergraph $\mathcal{H}^*$ with $\|\mathcal{H}^*\|_{j,p}=\binom{n}{j}^{1/p}\binom{u-j}{r-j}$ such that 
$ k_t^{(r)}(\mathcal{H}^*) = \frac{\binom{n}{j}}{\binom{u}{j}} \binom{u}{t}$.
\end{proposition}

The rest of this paper is organized as follows. Section \ref{sec:pre} introduces the necessary notation and entropy preliminaries. Section \ref{sec:upp} prove  Theorem \ref{thm:hypergraph_fixed_vertex} using the entropy method. Section \ref{sec:low} presents the extremal constructions using Steiner systems to establish Proposition \ref{prop:exact_construction}.

%%%%%%%%%%%%%%%%%%%%%%%%%%%%%%%%%%%%%%%%%%%%%%%%%%%%%%%%%%%%%%%%%%%%%%%%%%%%%%%%%%%%

%%%%%%%%%%%%%%%%%%%%%%%%%%%%%%%%%%%%%%%%%%%%%%%%%%%%%%%%%%%%%%%%%%%%%%%%%%%%%%%%%%%%%%%%%%%%%
%
\section{Preliminaries} \label{sec:pre}
We prove Theorem \ref{thm:hypergraph_fixed_vertex} using the entropy approach introduced  by Chao--Yu~\cite{CY24}. 
This section outlines the entropy properties used throughout the paper. Let $X$ be a discrete random variable taking values in a finite set $\Omega$. For simplicity, let $p_X(x) := \mathbb{P}(X=x)$ for every $x \in \Omega$. The \emph{support} of $X$ is defined as 
\begin{equation*}
    \mathrm{supp(X)}\coloneqq\{x\in\Omega: p_X(x)>0\}.
\end{equation*}
The \emph{Shannon entropy} \cite{S48} of $X$ is defined as 
\begin{equation*}
    \mathbb{H}(X)\coloneqq -\sum\limits_{x\in \mathrm{supp}(X)}p_X(x)\cdot \log_2{p_X(x)}.
\end{equation*}
The following fundamental upper bound on entropy is widely used.
\begin{proposition}\label{prop:bound}
For any discrete random variable $X$, $\mathbb{H}(X) \le \log_2|\mathrm{supp}(X)|$, where equality holds if and only if the distribution of $X$ is uniform on $\mathrm{supp}(X)$.
\end{proposition}

For any random vector $\mathbf{X} = (X_1, \dots, X_n)$, we denote its \emph{joint entropy} by $\mathbb{H}(\mathbf{X}) := \mathbb{H}(X_1, \dots, X_n)$. For two random variables $X$ and $Y$, let $X \mid Y = y$ denote the \emph{conditional random variable of $X$ given $Y = y$}, whose entropy is denoted by $\mathbb{H}(X \mid Y = y)$. The \emph{conditional entropy of $X$ given $Y$} is defined as 
\begin{align*}
    \mathbb{H}(X \mid Y) := \sum_{y \in \mathrm{supp}(Y)} p_Y(y) \mathbb{H}(X \mid Y = y).
\end{align*}
One can easily verify that $\mathbb{H}(X \mid Y) = \mathbb{H}(X, Y) - \mathbb{H}(Y)$, which yields the following chain rule.
\begin{proposition}\label{prop:chain}
Let $X_1, \dots, X_n$ be discrete random variables. Then
$$\mathbb{H}(X_1,\dots,X_n)=\mathbb{H}(X_1)+\mathbb{H}(X_2 \mid X_1)+\dots+\mathbb{H}(X_n \mid X_1,\dots,X_{n-1}).$$
\end{proposition}

Furthermore, we will rely on the following crucial lemma regarding the entropy of randomly ordered subsets.
\begin{lemma}[Chao--Dong--Shen--Yang \cite{CDSY24}] \label{lem:cdsy}
     Let $\mathcal{A}$ be a family of $d$-subsets of $[n]$. We sample a set $A \in \mathcal{A}$ uniformly at random, and then uniformly choose a random ordering of the elements of $A$ to form a vector $(X_1, \dots, X_d) \in [n]^d$. Then
\[
2^{\mathbb{H}(X_1)} \geq 2^{\mathbb{H}(X_2|X_1)} + 1 \geq 2^{\mathbb{H}(X_3|X_1, X_2)} + 2 \geq \dots \geq 2^{\mathbb{H}(X_d|X_1, \dots, X_{d-1})} + d - 1.
\]
\end{lemma}
\section{Proof of Theorem \ref{thm:hypergraph_fixed_vertex}} \label{sec:upp}

In this section, we prove Theorem~\ref{thm:hypergraph_fixed_vertex} using ideas from~\cite{CDSY24} with  a continuous interpolation step. Let $t, r, j, n$ be fixed integers such that $n\ge t > r > j \ge 1$, and let  $C, p>0$ be real numbers satisfying 
\[
p > \frac{t-j}{r-j} \quad \text{ and } \quad C^p \ge \binom{n}{j} \binom{\tilde{s}_{\mathbb{R}}-j}{r-j}^p. 
\]
Let  $\mathcal{H}$ be an $n$-vertex $r$-graph with $\|\mathcal{H}\|_{j,p}\le C$.  Suppose for the sake of contradiction that $k_t^{(r)}(\mathcal{H}) > \frac{\binom{n}{j}}{\binom{u}{j}}\binom{u}{t}$. Let $K$ be a $t$-clique chosen uniformly at random from all $t$-cliques of  $\mathcal{H}$, and let $(X_1, \dots, X_t)$ be a uniform random permutation of $V(K)$. 

For $k\in[t]$, set 
\[
x_k \coloneqq 2^{\mathbb{H}(X_k \mid X_1, \dots, X_{k-1})}.
\]
It follows from  Propositions \ref{prop:bound} and \ref{prop:chain} that  
\begin{equation}
x_1 x_2 \cdots x_t = t! k_t^{(r)}(\mathcal{H}) > t! \frac{\binom{n}{j}}{\binom{u}{j}} \binom{u}{t} = j!\binom{n}{j} (u-j)(u-j-1)\cdots(u-t+1). \label{eq:total_product}
\end{equation}

Since $(X_1, \dots, X_j)$ is an ordered $j$-tuple of distinct vertices in $V$, the cardinality of its support is at most $n!/(n-j)!$. By Proposition \ref{prop:bound}, we have 
\begin{equation}
x_1 \cdots x_j = 2^{\mathbb{H}(X_1, \dots, X_j)} \le n(n-1)\cdots(n-j+1) = j!\binom{n}{j}. \label{eq:prefix_bound}
\end{equation}
Dividing \eqref{eq:total_product} by \eqref{eq:prefix_bound}, we obtain the tail bound
\begin{align}\label{ine:tail_bound}
    x_{j+1} x_{j+2} \cdots x_t > (u-j)(u-j-1)\cdots(u-t+1).
\end{align}

To decouple the constraints, we partition the sequence into three segments: the prefix $A := x_1 \cdots x_j$, the extension $B := x_{j+1} \cdots x_r$, and the tail $C_{\mathrm{tail}} := x_{r+1} \cdots x_t$. The following claim shows that $AB^p$ can be bounded in terms of the $j$-degrees of $\mathcal{H}$.

\begin{claim}\label{cla:AB_upper}
 We have    
 \[
 AB^p \le j!(r - j)!^p \sum\limits_{s \in \binom{V}{j}} \deg(s)^p.
 \]
\end{claim}
\begin{proof}
 Let $\vec{S} = (X_1, \dots, X_j)$ be the ordered prefix, and let $S \in \binom{V}{j}$ be its underlying set. By 
Proposition \ref{prop:chain}, we have  
\begin{align}\label{EQ:another-form-ABp-THM}
AB^p=2^{\mathbb{H}(\vec{S}) + p\mathbb{H}(X_{j+1},\dots,X_r | \vec{S})}. 
\end{align}
Now, we bound $\mathbb{H}(X_{j+1},\dots,X_r | \vec{S})$. Note that for a  $t$-clique $K$ in $\mathcal{H}$, any $r$-subset of $V(K)$ must be an edge in $\mathcal{H}$. 
Consequently, for any valid prefix $\vec{s} \in \text{supp}(\vec{S})$, the subsequent sequence $(X_{j+1}, \dots, X_r)$ together with $s$ forms an $r$-edge in $\mathcal{H}$, and 
the number of such extensions is at most $(r - j)!\deg(s)$. By Proposition \ref{prop:bound},
\begin{align*}
    \mathbb{H}(X_{j+1}, \dots, X_r \mid \vec{S} = \vec{s}) \le \log_2((r - j)!\deg(s)).
\end{align*}
Taking the expectation over all possible prefixes $\vec{s} \in \text{supp}(\vec{S})$ yields that
\begin{align*}
    \mathbb{H}(X_{j+1}, \dots, X_r \mid \vec{S}) \le \sum_{\vec{s}\in \mathrm{supp}({\vec{S}})} \mathbb{P}(\vec{s}) \log_2((r - j)!\deg(s)).
\end{align*}

Recall the definition of entropy $\mathbb{H}(\vec{S}) = \sum\limits_{\vec{s}\in \mathrm{supp}({\vec{S}})} \mathbb{P}(\vec{s}) \log_2 \frac{1}{\mathbb{P}(\vec{s})}$. Applying Jensen's inequality to the concave function $\log_2(x)$ yields
\begin{align}\label{EQ:upper-bound-ABp-THM}
    \mathbb{H}(\vec{S}) + p\mathbb{H}(X_{j+1}, \dots, X_r \mid \vec{S}) &\le \sum_{\vec{s}\in \mathrm{supp}({\vec{S}})} \mathbb{P}(\vec{s}) \log_2 \left( \frac{((r - j)!\deg(s))^p}{\mathbb{P}(\vec{s})} \right) \notag\\
    &\le \log_2 \left( \sum_{\vec{s}\in \mathrm{supp}({\vec{S}})} ((r - j)!\deg(s))^p \right).
\end{align}
For each underlying $j$-set $s$ appearing in the support of $S$, all $j!$ orderings of $s$ appear in $\mathrm{supp}(\vec S)$. Hence, grouping the terms by their underlying $j$-sets gives 
\begin{equation}\label{EQ:underlying-support-THM}
\sum_{\vec{s}\in \mathrm{supp}(\vec S)} ((r-j)!\deg(s))^p = j!(r-j)!^p \sum_{s\in \mathrm{supp}(S)}\deg(s)^p.
\end{equation}
Since $\deg(s)^p \ge 0$, extending this sum to all subsets in $\binom{V}{j}$ can only increase the right-hand side.  Combining \eqref{EQ:another-form-ABp-THM}, \eqref{EQ:upper-bound-ABp-THM} and \eqref{EQ:underlying-support-THM} yields
\begin{equation*}
    AB^p=2^{\mathbb{H}(\vec{S}) + p\mathbb{H}(X_{j+1},\dots,X_r | \vec{S})} \le j!(r - j)!^p \sum_{s \in \binom{V}{j}} \deg(s)^p,  
\end{equation*}
completing the proof of Claim \ref{cla:AB_upper}. \qedhere
\end{proof}
We derive a conflicting lower bound on $AB^p$ by applying our continuous interpolation to the tail.
\begin{claim} \label{cla:AB_lower}
    $AB^p > j!\binom{n}{j}((u-j)\cdots(u-r+1))^p.$
\end{claim}
\begin{proof}
Rewriting the entropy product, it follows from \eqref{eq:total_product} that 
\begin{align}\label{EQ:main-ineq-ABp-THM}
    AB^p = (ABC_{\text{tail}}) \frac{B^{p-1}}{C_{\text{tail}}} > \left[ j!\binom{n}{j} (u-j)\cdots(u-t+1) \right] \frac{B^{p-1}}{C_{\text{tail}}}.
\end{align}
To complete the proof, it suffices to give a lower bound for  $\frac{B^{p-1}}{C_{\text{tail}}}$.

Since entropy is non-negative, we naturally have $x_k \ge 1$ for all $k \in [t]$, which guarantees $B \ge 1 $. 
Consider the function $f_j(y)=\prod_{i=j}^{r-1}(y - i)$. Clearly, $f_j(y)$ is strictly  increasing on the interval $[r - 1, \infty)$ with $f_j(r-1)=0$.
There exists a unique $x' \in (r-1, \infty)$ such that 
\begin{align}\label{eq:B}
    B = \prod_{i=j}^{r-1}(x' - i).
\end{align} 

We claim that 
\begin{align}\label{EQ:bound-xr-with-x'-THM}
x_r \le x' - r + 1. 
\end{align}
Indeed, Lemma \ref{lem:cdsy} gives that 
\begin{align}\label{EQ:gap-xk-xk+1-THM}
x_k - x_{k+1} \ge 1 \text { for  } 1\le k< t. 
\end{align}
If  $x_r > x' - r + 1$, then by \eqref{EQ:gap-xk-xk+1-THM},  $x_k \ge x_r + r - k > x' - k + 1$ for every $j+1 \le k \le r$. Taking the product over these terms gives
\begin{align*}
    B = \prod_{k=j+1}^r x_k>\prod_{k=j+1}^r (x'-k+1)=\prod_{i=j}^{r-1}(x'-i),
\end{align*}
which  contradicts \eqref{eq:B}.

Combining \eqref{EQ:bound-xr-with-x'-THM} and \eqref{EQ:gap-xk-xk+1-THM}, we have $x_k\le x'-k+1$ for all $r\le k\le t$. This together with $x_t\ge 1$ yields that $x'\ge t$ and then  
\begin{align} \label{ine:C_tail_bound}
    C_{\text{tail}}=x_{r+1}\cdots x_t \le \prod_{i=r}^{t-1}(x'-i).
\end{align}
By  \eqref{ine:tail_bound}, \eqref{eq:B} and \eqref{ine:C_tail_bound}, we conclude that 
\begin{align*}
    \prod_{i=j}^{t-1}(u-i) < B C_{\text{tail}} \le  \prod_{i=j}^{t-1}(x'-i).
\end{align*}
Since the function $g_j(y)=\prod_{i=j}^{t-1}(y-i)$ is strictly increasing on $(t-1,\infty)$ and $x',u>t-1$, the above inequality implies that $x'>u\ge \tilde{s}_R$. 
Since $p>\frac{t-j}{r-j}>1$, using \eqref{eq:B} and
\eqref{ine:C_tail_bound}, we obtain
\begin{align*}
    \frac{B^{p-1}}{C_{\text{tail}}} \ge \frac{((x'-j)\cdots(x'-r+1))^{p-1}}{(x'-r)\cdots(x'-t+1)} = \frac{j!(r-j)!^p}{t!} \cdot \frac{1}{\tilde{h}(x')}.
\end{align*}

Recall that  $\tilde{h}(x)$ is strictly decreasing on
$[\widetilde s_R,\infty)$. It follows from  $x'>u\ge \widetilde s_R$ that 
\begin{align*}
    \frac{B^{p-1}}{C_{\text{tail}}} > \frac{((u-j)\cdots(u-r+1))^{p-1}}{(u-r)\cdots(u-t+1)}.
\end{align*}

Substituting this into \eqref{EQ:main-ineq-ABp-THM} establishes that 
\begin{equation*}
    AB^p > j!\binom{n}{j}((u-j)\cdots(u-r+1))^p, 
\end{equation*}
completing the proof of Claim \ref{cla:AB_lower}. \qedhere
\end{proof}

Combining Claims \ref{cla:AB_upper} and  \ref{cla:AB_lower}, we thus conclude that
\begin{align*}
    j!(r - j)!^p \sum_{s \in \binom{V}{j}} \deg(s)^p > j!\binom{n}{j}((u-j)\cdots(u-r+1))^p.
\end{align*}
Dividing both sides by $j!(r - j)!^p$, we obtain
\begin{align*}
    \sum_{s \in \binom{V}{j}} \deg(s)^p > \binom{n}{j} \left( \frac{(u-j)\cdots(u-r+1)}{(r-j)!} \right)^p = \binom{n}{j}\binom{u-j}{r-j}^p =C^p.
\end{align*}

This contradicts the norm constraint $\sum\limits_{s \in \binom{V}{j}} \deg(s)^p \le C^p$. This completes the proof of Theorem \ref{thm:hypergraph_fixed_vertex}.

%

%%%%%%%%%%%%%%%%%%%%%%%%%%%%%%%%%%%%%%%%%%%%%%%%%%%%%%%%%%%%%%%%%%%%%%%%%%%%%

%%%%%%%%%%%%%%%%%%%%%%%%%%%%%%%%%%%%%%%%%%%%%%%%%%%%%%%%%%%%%%%%%%%%%%%%%%
\section{Extremal Constructions via Steiner Systems}\label{sec:low}
In this section, we construct the extremal $r$-graph $\mathcal{H}^*$ that exactly achieves the upper bound in Theorem \ref{thm:hypergraph_fixed_vertex}. The extremal structure  is a highly symmetric combinatorial design based on a Steiner system, which naturally unifies and generalizes the disjoint union of cliques known for the $j = 1$ case.

\begin{proof}[Proof of Proposition \ref{prop:exact_construction}]

Fix $j$ and $u$, and let $n$ be a sufficiently large integer satisfying~\eqref{necessary-divisibility-condition-n}. By a result of Keevash~\cite{K14}, there exists a Steiner system $S(j, u, n)=(V, \mathcal{B})$ such that each block $B \in \mathcal{B}$ has size $u$, and every $j$-subset $S \in \binom{V}{j}$ is contained in exactly one block.  We construct the $r$-graph $\mathcal{H}^*$ on $V$ by replacing each block in $(V, \mathcal{B})$ with a complete $r$-graph; that is, the edge set $ E^*$ of $\mathcal{H}^*$ is 
    \begin{align*}
        E^* = \bigcup_{B \in \mathcal{B}} \binom{B}{r}.
    \end{align*}

    First, we verify the $(j, p)$-norm of $\mathcal{H}^*$. For any $S \in \binom{V}{j}$, let $B_S \in \mathcal{B}$ denote the unique block containing $S$. By the definition of $E^*$, the set of edges containing $S$ is entirely restricted to $B_S$: 
    \begin{align*}
        \{e \in E^* : S \subseteq e\} = \left\{e \in \binom{B_S}{r} : S \subseteq e \right\}.
    \end{align*}
    Thus, the degree of every $j$-subset is exactly $\deg_{\mathcal{H}^*}(S) = \binom{|B_S|-j}{r-j} = \binom{u-j}{r-j}$. Evaluating the norm yields
    \begin{align*}
        \|\mathcal{H}^*\|_{j, p}^p = \sum_{S \in \binom{V}{j}} \deg_{\mathcal{H}^*}(S)^p = \sum_{S \in \binom{V}{j}} \binom{u-j}{r-j}^p = \binom{n}{j}\binom{u-j}{r-j}^p.
    \end{align*}

Next, we evaluate the number of $t$-cliques in $\mathcal{H}^*$ by  showing that for every $t$-clique $K$ in $\mathcal{H}^*$, there exists a unique block $B \in \mathcal{B}$ such that $K \subseteq B$.Suppose for contradiction that there exists a $t$-clique $K$ such that $K \not\subseteq B$ for all $B \in \mathcal{B}$. Fix $S \in \binom{K}{j}$, and let $B_a \in \mathcal{B}$ be the unique block containing $S$. Since $K \not\subseteq B_a$, we can choose $v \in K \setminus B_a$. Given $t \ge r \ge j+1$, there exists an $r$-subset $e$ satisfying
    \begin{align*}
        S \cup \{v\} \subseteq e \subseteq K.
    \end{align*}
    Since $K$ is a clique, $e \in E^*$. By the construction of $E^*$ and the Steiner property, there exists a unique block $B_b \in \mathcal{B}$ such that $e \subseteq B_b$. This establishes the following chain of inclusions
    \begin{align*}
        S \subseteq S \cup \{v\} \subseteq e \subseteq B_b.
    \end{align*}
    However, $S \subseteq B_b$ uniquely determines $B_b = B_a$. This implies $v \in e \subseteq B_a$, directly contradicting $v \in K \setminus B_a$.

    Consequently, the set of $t$-cliques in $\mathcal{H}^*$ is partitioned by the blocks in $\mathcal{B}$. Since $|\mathcal{B}| = \binom{n}{j}/\binom{u}{j}$ , we have 
    \begin{equation*}
        k_t^{(r)}(\mathcal{H}^*) = \sum_{B \in \mathcal{B}} k_t^{(r)}(B) = |\mathcal{B}|\binom{u}{t} = \frac{\binom{n}{j}}{\binom{u}{j}}\binom{u}{t}.
    \end{equation*}
    This completes the proof of Proposition \ref{prop:exact_construction}.
\end{proof}

\section{Concluding remarks}

In this note, we  answer the fixed-vertex question raised by Chao--Dong--Shen--Yang in the remaining range
\[
    p>\frac{t-j}{r-j}
    \quad\text{and}\quad
    C^p \ge \binom{n}{j}\binom{\widetilde s_{\mathbb R}-j}{r-j}^p .
\]
The proof isolates the new difficulty caused by the fixed vertex set in the hypergraph setting: the entropy of the initial $j$-set must be coupled with the local $j$-degree norm. This is handled by splitting a random ordered clique into a $j$-prefix, an $(r-j)$-extension, and a tail, and by applying a continuous interpolation step.

The upper bound is sharp whenever the corresponding Steiner systems exist. It would be interesting to understand the exact extremal configurations in the non-divisible cases, and to establish corresponding stability results. Another possible direction is to study analogous counting problems for other fixed hypergraphs under $(j,p)$-norm constraints.

\bibliographystyle{abbrv}
\bibliography{cliqueHypergraph}
\end{document}